\documentclass[11pt]{article}

\usepackage[a4paper,margin=28mm]{geometry}
\usepackage[authoryear,round]{natbib}
\newcommand{\papertitle}{Geometry and observability of latent texture–coherence diffusions in coherent sea-clutter radar observations}

\newcommand{\paperkeywords}{stochastic differential equations; sea clutter; partial observation; diffusion geometry; quadratic variation; observability; radar}

\newcommand{\paperabstract}{%
A coherent radar observes sea clutter through a non-injective multiplicative combination of latent texture and coherence variables. We study a class of partially observed slow--fast diffusions for which the projected quadratic covariance has the form \(Q=S+qyy^\top\), using Field's sea-clutter diffusion as a solvable benchmark. For fixed-shape fast covariance, we derive an exact generalized spectral-gap criterion for recovery of the latent slow coordinate and a stability bound. When the fast covariance changes shape, we obtain an invariant scale--shape decomposition and exact conditions for reinforcement or cancellation of the latent quadratic signature. We then quantify the loss of resolvability caused by finite sampling and receiver noise. Controlled CIR--OU simulations verify the analytic identities and exhibit regimes in which structural observability does not yield stable reconstruction.%
}

\usepackage{amsmath,amssymb,mathtools,bm,amsthm}
\usepackage{booktabs}
\usepackage{graphicx}
\usepackage{xcolor}
\usepackage{hyperref}
\usepackage[nameinlink,noabbrev]{cleveref}

\numberwithin{equation}{section}

\providecommand{\E}{\mathbb{E}}

\providecommand{\Cov}{\operatorname{Cov}}
\providecommand{\Var}{\operatorname{Var}}
\providecommand{\tr}{\operatorname{tr}}
\providecommand{\diag}{\operatorname{diag}}
\providecommand{\SPD}{\operatorname{SPD}}
\providecommand{\Sym}{\operatorname{Sym}}

\providecommand{\dd}{\,\mathrm{d}}

\newif\ifdraftnotes
\draftnotestrue

\newtheorem{theorem}{Theorem}[section]
\newtheorem{proposition}[theorem]{Proposition}
\newtheorem{corollary}[theorem]{Corollary}

\title{\papertitle}
\author{%
Arnaud Coatanhay \quad and \quad Ang\'elique Dr\'emeau\\[0.6em]
\small \textit{Lab-STICC, UMR CNRS 6285, ENSTA, Institut Polytechnique de Paris}\\
\small \textit{2 rue Fran\c{c}ois Verny, 29806 Brest Cedex 9, France}\\[0.3em]
\small \texttt{arnaud.coatanhay@ensta.fr}\quad
\small \texttt{angelique.dremeau@ensta.fr}
}
\date{\today}

\begin{document}
\maketitle

\begin{abstract}
\paperabstract
\end{abstract}


\section{Introduction}\label{sec:introduction}

\subsection{Dynamic sea clutter and coherent observation}

Sea-surface backscatter has been a central limitation in maritime radar since the earliest surveillance systems, while the same return also carries information on the ocean surface.  Its characteristics depend strongly on acquisition geometry, grazing angle, carrier frequency, polarization, spatial resolution and sea state \citep{WardBakerWatts1990Maritime,WardToughWatts2013SeaClutter,WattsRosenberg2022Challenges}.  High-resolution measurements further show that the amplitude statistics and correlation structure can depart markedly from simple homogeneous Gaussian models \citep{FarinaGiniGrecoVerrazzani1997HighResolution}, and coherent monostatic/bistatic measurements exhibit geometry- and polarization-dependent texture and Doppler behaviour \citep{RitchieStoveWoodbridgeGriffiths2016NetRAD}.  These observations motivate a stochastic description that keeps the acquisition dependence visible rather than treating ``sea clutter'' as a universal stationary noise law.

A coherent radar does not observe the sea surface directly.  It records a complex electromagnetic return whose amplitude and phase result from the superposition of many scattering contributions inside a resolution cell.  At the signal level, compound and random-walk descriptions have proved particularly useful for representing non-Gaussian fluctuations \citep{WardToughWatts2013SeaClutter}.  The K distribution is a canonical example: it arises from a fluctuating-step random walk and has a long history in scattering experiments \citep{JakemanPusey1978KDistribution}.  Such one-time distributional models describe amplitudes or intensities at a fixed instant, but they do not by themselves determine multi-time moments, transition probabilities, innovations, trajectory likelihoods, or the information carried by a time series about latent dynamical variables.

This distinction becomes operational whenever coherent measurements are combined across time.  Ocean-surface motion can, for example, alter coherent SAR image formation \citep{AlpersBruening1986SAR}; more generally, prediction, synchronization of measurements acquired at different instants, dynamical parameter estimation and coherent target/background discrimination all depend on temporal structure.  Continuous-time diffusion models were introduced into the K-distributed scattering framework by \citet{FieldTough2003Diffusion,FieldTough2003Stochastic} and subsequently specialized to sea clutter by \citet{FieldHaykin2008NonlinearDynamics,Field2008Book}.  The developments of \citet{RousselCoatanhayBaussard2019Parameters,RousselCoatanhayBaussard2019ForwardBackward,RousselCoatanhayBaussard2021CoherentScattering} then exploited this explicit transition structure for inference, parameter estimation and coherent-scattering problems.

Field's construction is particularly attractive because it retains the familiar separation between a slowly varying positive modulation and a rapidly decorrelating complex speckle while assigning a stochastic dynamics to both components.  In its simplest form the coherent reflectivity is factorized as
\begin{equation}
  \Psi_t=\sqrt{x_t}\,\gamma_t,
  \label{eq:intro-field-factorization}
\end{equation}
where the positive process $x_t$ modulates the local return power and the complex process $\gamma_t$ carries the fast coherent fluctuations.  This factorization turns a static compound model into a genuine multi-time stochastic model.

The present paper starts from this dynamical viewpoint, but its objective is not to defend Field's model as a complete microscopic description of the sea surface.  Recent reviews continue to emphasize the difficulty of constructing sea-clutter models that remain robust across sea conditions, viewing geometries, resolutions and coherent processing modes \citep{WattsRosenberg2022Challenges}.  Field's model itself relies on idealizations---independent elementary scatterers, a diffusion description of phase decorrelation, a large-population limit, and no explicit cyclic wave variable---that need not remain valid at all resolutions or sea states.  We therefore use Field's diffusion as a \emph{canonical solvable benchmark}: a model simple enough to permit exact calculations, but rich enough to expose the interaction between latent dynamics, observation and inference.

\subsection{The observation problem}

The main difficulty considered here is created by the observation itself.  Even if the latent state is Markov and its stochastic differential equation is known, the radar does not separately observe texture and speckle.  At a given instant, the mapping
\[
 (x,\gamma)\longmapsto \Psi=\sqrt{x}\,\gamma
\]
is many-to-one.  An observed complex value therefore specifies a fibre of latent states rather than a unique state.  This is not a secondary estimation issue: it changes the structure of the stochastic process seen by the radar.

The earlier sea-clutter work already encountered this difficulty operationally.  Parameter estimation can be carried out straightforwardly when the positive modulation is treated as observable, whereas a practical procedure requires reconstructing latent quantities from the measured reflectivity \citep{RousselCoatanhayBaussard2019Parameters}.  Similarly, forward and backward probabilistic inference is naturally formulated in the latent transition structure even though the measured process is a projection of it \citep{RousselCoatanhayBaussard2019ForwardBackward}.  The coherent-scattering extension further shows that physically meaningful perturbations may enter either as an additive observed component or through a state-dependent coupling \citep{RousselCoatanhayBaussard2021CoherentScattering}.  These developments lead to a structural question: which features of a latent
texture--coherence diffusion remain visible after coherent radar observation, and how can
that visibility be characterized geometrically?  Injectivity of the instantaneous observation
map is not sufficient to answer it.  A continuous trajectory contains additional local information through its drift and, crucially for the present work, through its quadratic covariation.  For a diffusion $X_t$ observed through a smooth map $h$, the projected quadratic tensor
\begin{equation}
  Q=Dh\,a_{\rm lat}\,Dh^\top
  \label{eq:intro-Q}
\end{equation}
is the infinitesimal covariance of the observed increments.  If $Q$ varies along a fibre of $h$, then the projected diffusion symbol does not descend to the observation space.  The same variation may, however, provide a trajectory-level signature of the latent position on that fibre.  This dual role is the central mechanism studied below.

The generic ingredients are not new.  Markov functions and lumpability have a long history \citep{RogersPitman1981MarkovFunctions,Coxson1984LumpabilityObservability,HorstmeyerAtay2016ExactLumpability}; quadratic variation can reveal hidden-state information under singular or signal-dependent observation structures \citep{JoannidesLeGland1997PerfectObservation,CrisanKouritzinXiong2009SignalDependentObservation}; and recent work gives strong identification results for diffusions from imperfect observations \citep{ClarkCrisan2026Identification}.  Our aim is therefore narrower.  We study a specific texture--coherence architecture and quantify how the same projected covariance tensor encodes (i) failure of stochastic reduction, (ii) latent visibility along observation fibres, and (iii) the loss of that visibility when the fast covariance changes shape or when the receiver adds a competing anisotropy.

\subsection{From a radar model to a geometric model class}

To separate what is structural from what is specific to the CIR--OU realization of Field's model, we introduce a slow--fast latent class
\[
  X_t=(R_t,Z_t),
\]
where $R_t$ is a scalar positive or Lamperti-transformed texture coordinate and $Z_t\in\mathbb R^m$ is a fast coherence variable.  The observation has the multiplicative coherent form
\begin{equation}
  Y_t=a(R_t)Z_t.
  \label{eq:intro-general-observation}
\end{equation}
The fast diffusion covariance $D(R)$ is allowed to depend on the slow variable.  This dependence leads to three qualitatively different situations: no coupling, pure scale coupling, and shape-changing coupling.  The distinction is natural both geometrically and physically.  A scalar change of fast variance modifies correlation scales without changing principal directions, whereas a change of shape can alter anisotropy, rotate principal axes, and interfere with the radial signature induced by the multiplicative observation.

Positive-definite matrix geometry provides the appropriate language for this scale--shape split \citep{Moakher2005SPD,JungSchwartzmanGroisser2015ScalingRotation,MostajeranSepulchre2018OrderingSPD}.  We use it here not to introduce a new geometry on $\SPD(m)$, but to formulate an invariant measure of shape variation and to connect that variation to what the observer can or cannot distinguish.

\subsection{Results and scope}

For the fixed-shape class \(D(r)=d(r)D_0\), a single whitening transformation reduces the
projected covariance to an isotropic term plus a rank-one contribution in the observed
direction.  The generalized longitudinal--transverse eigenvalue gap then isolates a scalar
function \(q(r)\) of the latent slow coordinate independently of the fast scale factor.  We
derive an exact recovery criterion and a local stability bound, with Field's CIR--OU model as
an explicit example.

When the fast covariance changes shape, the vertical derivative of the observed quadratic
tensor admits an invariant decomposition into scale and traceless shape components.  The
decomposition gives a lower bound on the surviving quadratic signature and characterizes
exact cancellation by shape variation.  It also shows that a non-zero derivative of the
multiplicative channel does not by itself guarantee local quadratic visibility.  The log-OU
example gives a separate limitation: the same multiplicative observation law need not yield
an informative spectral texture coordinate.

The ideal continuous-time results are then compared with finite-rate acquisition.  Local
covariance estimation introduces a window-length trade-off, receiver noise and possible
instrumental anisotropy.  Pre-averaging and realized-kernel methods provide the relevant
statistical background
\citep{JacodEtAl2009Preaveraging,BarndorffNielsenEtAl2008RealizedKernels,BarndorffNielsenEtAl2011MultivariateKernels,ParkHongLinton2016FourierKernel}.
Controlled simulations show that structural identifiability does not necessarily imply
accurate recovery over a finite observation window.

The analysis is restricted to information carried by the diffusion symbol and its coherent
observation.  Models with the same diffusion covariance but different drifts cannot be
distinguished at this level.  Nor is the positive latent coordinate interpreted as a fundamental
state variable of the sea surface; in a more physical model the effective radar cross section
would depend on an underlying sea state and on the acquisition configuration.  The historical
radar data associated with the earlier programme are not available, so the numerical results
are confined to controlled simulations and misspecification tests.

The remainder of the paper is organized as follows.  \Cref{sec:model-class} introduces the general latent texture--coherence class and fixes the three modelling levels---physical state, reduced stochastic state and radar observation.  \Cref{sec:field} then states the Field benchmark and its time scales.  \Cref{sec:symbol-geometry,sec:shape-geometry} analyze the latent diffusion geometry and the scale--shape structure of the fast covariance.  \Cref{sec:partial-observation} formulates projectability and the vertical quadratic defect.  The two central results are given in \Cref{sec:fixed-shape,sec:shape-changing}.  Finite-sampling and receiver effects are treated in \Cref{sec:resolvability}, followed by controlled numerical tests in \Cref{sec:numerics}.

\section{Latent texture--coherence diffusion class}\label{sec:model-class}

\subsection{Physical, latent and observed levels}

We distinguish the physical sea state from the reduced stochastic variables used for inference.
Let \(S_t\) denote an effective state of the illuminated surface and let \(\mathfrak a\) collect
the acquisition parameters. An effective scattering strength should then be regarded as an
acquisition-dependent response
\begin{equation}
  x_t^{(\mathfrak a)}=\sigma_{\mathfrak a}(S_t),
  \label{eq:physical-to-effective}
\end{equation}
rather than as a universal microscopic coordinate of the sea. The paper does not model
\(S_t\) explicitly.

The reduced latent state is
\begin{equation}
  X_t=(R_t,Z_t),\qquad Z_t\in\mathbb R^m,
  \label{eq:latent-state}
\end{equation}
with a scalar slow variable \(R_t\) and a fast coherence variable \(Z_t\). The coherent
observation is
\begin{equation}
  Y_t=h(R_t,Z_t)=a(R_t)Z_t,\qquad a>0.
  \label{eq:coherent-observation}
\end{equation}
Thus the mathematical problem concerns the map from the reduced latent diffusion to the
coherent observation; the link from the physical sea state to the reduced variables remains
model dependent.

\subsection{Slow--fast diffusion}

Suppose locally that
\begin{equation}
  \dd R_t=b(R_t)\dd t+\sigma(R_t)\dd W_t^0,\qquad \sigma>0.
  \label{eq:slow-original}
\end{equation}
After the Lamperti change \(r=\ell(R)\), \(\ell'=1/\sigma\), we write
\begin{align}
  \dd r_t &= \beta(r_t)\dd t+\dd W_t^0,
  \label{eq:slow-lamperti}\\
  \dd Z_t &= c(r_t,Z_t)\dd t+\Sigma(r_t)\dd W_t,
  \label{eq:fast-general}
\end{align}
where
\begin{equation}
  D(r)=\Sigma(r)\Sigma(r)^\top\in\SPD(m),
  \label{eq:fast-covariance}
\end{equation}
and
\begin{equation}
  \dd[W^0,W]_t=0.
  \label{eq:block-diagonal-noise}
\end{equation}
We assume the regularity required below, in particular \(a\in C^2\) and \(D\in C^1\), and
work where the latent symbol is non-degenerate. No fluctuation--dissipation relation is imposed
between \(c\) and \(D\): the former controls finite-time drift, the latter the instantaneous
covariance.

\subsection{Projected quadratic tensor}

For \(Y_t=a(r_t)Z_t\), It\^o's formula gives the diffusion part
\[
  a'(r)z\,\dd W_t^0+a(r)\Sigma(r)\dd W_t.
\]
On the observation fibre \(y=a(r)z\), the quadratic-covariation density is therefore
\begin{equation}
  Q(r,y)
  =
  a(r)^2D(r)
  +
  \left(\frac{a'(r)}{a(r)}\right)^2yy^\top .
  \label{eq:Q-master-model}
\end{equation}
With
\begin{equation}
  S(r)=a(r)^2D(r),
  \qquad
  q(r)=\bigl(\partial_r\log a(r)\bigr)^2,
  \label{eq:S-q-definitions}
\end{equation}
this becomes
\begin{equation}
  Q(r,y)=S(r)+q(r)yy^\top.
  \label{eq:Q-master}
\end{equation}
In the original coordinate \(R\),
\begin{equation}
  q_R(R)=\sigma(R)^2
  \left(\frac{a_R'(R)}{a(R)}\right)^2.
  \label{eq:q-original-coordinate}
\end{equation}

We call \(D(r)=d(r)D_0\) \emph{fixed-shape} or \emph{scale coupling}; otherwise the
normalized fast covariance changes shape. The main questions are whether \(r\mapsto Q(r,y)\)
separates points of a fixed observation fibre, how well conditioned the inverse is, and how
finite sampling and receiver noise alter this ideal information.

\section{Field sea clutter as a canonical solvable benchmark}\label{sec:field}

\subsection{Texture--speckle dynamics}

The Field model used in the earlier sea-clutter work combines a positive CIR texture with a
complex OU speckle \citep{FieldHaykin2008NonlinearDynamics,Field2008Book}:
\begin{equation}
  \Psi_t=\sqrt{x_t}\,\gamma_t,\qquad
  \gamma_t=\gamma_t^R+i\gamma_t^I,
  \label{eq:field-factorization}
\end{equation}
with independent Brownian drivers and
\begin{align}
 \dd x_t
 &=A(1-x_t)\dd t
   +\sqrt{\frac{2A}{\alpha}x_t}\,\dd W_t^x,
 \label{eq:field-cir}\\
 \dd\gamma_t^R
 &=-\frac B2\gamma_t^R\dd t
   +\sqrt{\frac B2}\,\dd W_t^R,
 \label{eq:field-ou-r}\\
 \dd\gamma_t^I
 &=-\frac B2\gamma_t^I\dd t
   +\sqrt{\frac B2}\,\dd W_t^I.
 \label{eq:field-ou-i}
\end{align}
At stationarity,
\begin{equation}
  x\sim\Gamma(\alpha,\text{rate }\alpha),
  \qquad
  \E x=1,\qquad \Var(x)=\alpha^{-1},
  \label{eq:field-gamma-stationary}
\end{equation}
and
\begin{equation}
  \gamma^R,\gamma^I\sim\mathcal N(0,1/2),
  \qquad
  \E|\gamma|^2=1.
  \label{eq:field-speckle-stationary}
\end{equation}
The characteristic relaxation times are \(A^{-1}\) for the texture mean and \(2/B\) for each
speckle quadrature. Exact transition densities underlie the earlier inference and
parameter-estimation results
\citep{RousselCoatanhayBaussard2019Parameters,RousselCoatanhayBaussard2019ForwardBackward}.

\subsection{Observation fibres and projected covariance}

In real coordinates \(u=\gamma^R\), \(v=\gamma^I\),
\begin{equation}
  h(x,u,v)=(\sqrt{x}\,u,\sqrt{x}\,v).
  \label{eq:field-observation-map}
\end{equation}
For \(\psi=(R,I)\neq0\),
\begin{equation}
  h^{-1}(\psi)
  =
  \left\{\left(x,\frac{R}{\sqrt{x}},\frac{I}{\sqrt{x}}\right):x>0\right\}.
  \label{eq:field-fibre}
\end{equation}
Thus one coherent sample does not determine texture and speckle separately.

The latent symbol is
\[
  a_{\rm lat}(x,u,v)
  =
  \diag\!\left(\frac{2Ax}{\alpha},\frac B2,\frac B2\right),
\]
and the observed covariance density is
\begin{equation}
  Q(x,\Psi)
  =
  \frac{xB}{2}I_2
  +
  \frac{A}{2\alpha x}\Psi\Psi^\top.
  \label{eq:field-Q}
\end{equation}
For \(\Psi\neq0\),
\begin{equation}
  \Lambda_\theta=\frac{xB}{2},
  \qquad
  \Lambda_\rho=\frac{xB}{2}
  +\frac{A|\Psi|^2}{2\alpha x},
  \label{eq:field-Q-eigenvalues}
\end{equation}
so
\begin{equation}
  \Lambda_\rho-\Lambda_\theta
  =
  \frac{A|\Psi|^2}{2\alpha x}.
  \label{eq:field-gap}
\end{equation}
These formulas already exhibit the mechanism developed abstractly in
\cref{sec:fixed-shape}: the instantaneous observation is non-injective, but its local
quadratic covariance varies along the fibre.

\subsection{Scope of the benchmark}

Field is useful here because its CIR and OU transitions are explicit and its multiplicative
observation permits exact calculations. We do not identify \(x_t\) with a universal physical
state of the sea. Coherent wave motion, intermittency, acquisition-dependent scattering and
receiver filtering are absent from the benchmark, and the historical data used in the earlier
work are not available for the present study. The model is therefore used as a controlled
reference generator, with alternative models introduced later to test the limits of the derived
structure.

\section{Diffusion geometry and scale coupling}
\label{sec:symbol-geometry}

\subsection{Diffusion symbol and metric}

For a diffusion
\begin{equation}
  \dd X_t^i=b^i(X_t)\dd t+\sigma^i_\alpha(X_t)\dd W_t^\alpha,
  \label{eq:generic-diffusion}
\end{equation}
write \(a=\sigma\sigma^\top\). Then
\[
  \Cov(\Delta X\mid X_t=x)=a(x)\Delta t+o(\Delta t),
\]
and
\begin{equation}
  2\Gamma(f,g)=a^{ij}\partial_i f\,\partial_j g,
  \qquad
  \frac{\dd}{\dd t}[f(X),g(X)]_t=2\Gamma(f,g)(X_t).
  \label{eq:carre-qv}
\end{equation}
When \(a\) is positive definite we use
\begin{equation}
  g=a^{-1}
  \label{eq:diffusion-metric}
\end{equation}
as the diffusion metric. It depends only on the symbol, not on the drift.

For Field,
\begin{equation}
  a
  =
  \diag\!\left(\frac{2Ax}{\alpha},\frac B2,\frac B2\right),
  \label{eq:field-symbol}
\end{equation}
hence
\begin{equation}
  \dd s^2
  =
  \frac{\alpha}{2Ax}\dd x^2
  +
  \frac2B(\dd u^2+\dd v^2).
  \label{eq:field-metric}
\end{equation}
The Lamperti coordinates
\begin{equation}
  r=\sqrt{\frac{2\alpha x}{A}},
  \qquad
  U=\sqrt{\frac2B}u,
  \qquad
  V=\sqrt{\frac2B}v
  \label{eq:field-lamperti}
\end{equation}
give
\begin{equation}
  \dd s^2=\dd r^2+\dd U^2+\dd V^2.
  \label{eq:field-euclidean}
\end{equation}
Thus the scalar Field symbol is flat on \(x>0\); the factor \(x^{-1}\) in
\eqref{eq:field-metric} is a coordinate effect. The transformed drift remains non-trivial,
so this statement concerns the symbol geometry rather than the complete generator.

\subsection{Scalar scale coupling}

Let \(B=B(r)>0\). Then
\begin{equation}
  g_B=\dd r^2+\frac2{B(r)}(\dd u^2+\dd v^2)
  \label{eq:warped-field}
\end{equation}
is a warped product. With
\begin{equation}
  q_B(r)=\log\frac{B(r)}{B_{\rm ref}},
  \label{eq:qB-def}
\end{equation}
the sectional curvatures are
\begin{align}
  K_{ru}=K_{rv}
  &=
  \frac12q_B''-\frac14(q_B')^2,
  \label{eq:K-radial}\\
  K_{uv}
  &=
  -\frac14(q_B')^2,
  \label{eq:K-fast}
\end{align}
and
\begin{equation}
  \mathcal R
  =
  2q_B''-\frac32(q_B')^2.
  \label{eq:scalar-curvature}
\end{equation}
Hence this family is flat if and only if \(B\) is constant. For
\(B(r)=B_0e^{\kappa r}\), all sectional curvatures equal \(K=-\kappa^2/4\).
Curvature detects non-constant scale coupling but not the sign of \(\kappa\).

\section{Shape coupling and positive-definite covariance geometry}
\label{sec:shape-geometry}

To allow the principal axes or relative eigenvalues of the fast covariance to vary with the
slow state, let
\[
  D(r)\in\SPD(m),
  \qquad
  S(r)=a(r)^2D(r).
  \label{eq:S-def}
\]
We use standard positive-definite matrix geometry
\citep{Moakher2005SPD,JungSchwartzmanGroisser2015ScalingRotation,MostajeranSepulchre2018OrderingSPD}
to separate scale from shape.

For symmetric tangent matrices \(H,K\) at \(S\), define
\begin{equation}
  \langle H,K\rangle_S
  =
  \tr(S^{-1}HS^{-1}K),
  \qquad
  \|H\|_S^2=\langle H,H\rangle_S,
  \label{eq:affine-relative-inner}
\end{equation}
and
\begin{equation}
  \Pi_S^0H
  =
  H-\frac{\tr(S^{-1}H)}{m}S.
  \label{eq:traceless-proj}
\end{equation}
For \(H=S'\),
\begin{equation}
  \kappa
  =
  \frac1m\tr(S^{-1}S')
  =
  \frac1m\frac{\dd}{\dd r}\log\det S,
  \label{eq:kappa-scale}
\end{equation}
measures logarithmic scale change, while
\begin{equation}
  \chi_{\rm sh}
  =
  \frac1{\sqrt2}\|\Pi_S^0S'\|_S
  \label{eq:chi-sh}
\end{equation}
measures shape change. On a connected interval,
\[
  \chi_{\rm sh}\equiv0
  \quad\Longleftrightarrow\quad
  S(r)=s(r)S_0
\]
for fixed \(S_0\in\SPD(m)\).

For \(m=2\),
\begin{equation}
  S(r)
  =
  s(r)R_{\theta(r)}
  \begin{pmatrix}
  e^{\eta(r)}&0\\
  0&e^{-\eta(r)}
  \end{pmatrix}
  R_{\theta(r)}^\top,
  \label{eq:S-2D}
\end{equation}
and
\begin{equation}
  \chi_{\rm sh}^2
  =
  (\eta')^2+4(\theta')^2\sinh^2\eta.
  \label{eq:chi-sh-2D}
\end{equation}
Shape can therefore change through eccentricity, rotation of the principal axes, or both.

The same distinction appears in the observer connection of the matrix-valued Field extension:
it is flat precisely for
\begin{equation}
  B(x)=b(x)B_0
  \label{eq:observer-flat-fixed-shape}
\end{equation}
with fixed \(B_0\in\SPD(2)\). Fixed anisotropy does not by itself produce observer holonomy;
only change of normalized shape does. Holonomy is not used as a recovery tool below, but this
equivalence motivates the fixed-shape/shape-changing split.

\section{Radar fibres, projectability and quadratic observability}
\label{sec:partial-observation}

\subsection{Observation fibres and covariance splitting}

For Field,
\begin{equation}
  h(x,u,v)=(R,I)=(\sqrt{x}\,u,\sqrt{x}\,v),
  \label{eq:field-observation-map-sec6}
\end{equation}
and, for \(\psi=(R,I)\neq0\),
\begin{equation}
  h^{-1}(\psi)
  =
  \left\{
  \left(x,\frac{R}{\sqrt{x}},\frac{I}{\sqrt{x}}\right):x>0
  \right\}.
  \label{eq:field-fibre-sec6}
\end{equation}
A vertical tangent is
\begin{equation}
  Z
  =
  \partial_x-\frac{u}{2x}\partial_u-\frac{v}{2x}\partial_v,
  \qquad Dh(Z)=0.
  \label{eq:vertical-field}
\end{equation}

The observed symbol is
\begin{equation}
  Q=Dh\,a_{\rm lat}\,Dh^\top
  =
  \frac{\dd[Y]_t}{\dd t}.
  \label{eq:projected-symbol-general}
\end{equation}
For scalar \(B=B(x)\),
\begin{equation}
  Q(x,\Psi)
  =
  \frac{xB(x)}2I_2
  +
  \frac{A}{2\alpha x}\Psi\Psi^\top.
  \label{eq:field-Q-observer}
\end{equation}
Rotational equivariance gives
\begin{equation}
  Q=\Lambda_\theta P_\theta+\Lambda_\rho P_\rho,
  \label{eq:Q-radial-transverse}
\end{equation}
with
\begin{equation}
  \Lambda_\theta=\frac{xB(x)}2,
  \qquad
  \Lambda_\rho
  =
  \frac{xB(x)}2+\frac{A|\Psi|^2}{2\alpha x},
  \label{eq:field-eigenvalues}
\end{equation}
and
\begin{equation}
  \Delta_Q
  =
  \Lambda_\rho-\Lambda_\theta
  =
  \frac{A|\Psi|^2}{2\alpha x}.
  \label{eq:field-gap-sec6}
\end{equation}

\subsection{Projectability and vertical defect}

A generator \(L\) is projectable through a submersion \(h:M\to N\) when
\begin{equation}
  L(F\circ h)=(\bar L F)\circ h
  \label{eq:generator-projectability}
\end{equation}
for every smooth \(F\). Since
\begin{equation}
  L(F\circ h)
  =
  (\partial_aF)(h)Lh^a
  +
  (\partial_{ab}F)(h)\Gamma(h^a,h^b),
  \label{eq:generator-chain-rule}
\end{equation}
projectability is equivalent to basicity of \(Lh^a\) and
\(\Gamma(h^a,h^b)\), the classical diffusion reduction/lumpability criterion
\citep{RogersPitman1981MarkovFunctions,Coxson1984LumpabilityObservability,HorstmeyerAtay2016ExactLumpability}.
For Field, \eqref{eq:field-Q-observer} varies with \(x\) at fixed \(\Psi\), so the quadratic
characteristic is not basic and the generator is not strongly projectable. This does not
exclude weaker Markov properties under special initial laws.

Let \(\mathcal V_x=\ker Dh_x\). For \(v\in\mathcal V_x\), choose a curve
\(c(s)\subset h^{-1}(h(x))\) with \(c(0)=x\), \(\dot c(0)=v\), and define
\begin{equation}
  d^{\rm v}Q_x(v)
  =
  \left.\frac{\dd}{\dd s}Q_{c(s)}\right|_{s=0}.
  \label{eq:vertical-defect}
\end{equation}

\begin{proposition}[Vertical criterion for symbol basicity]
\label{prop:vertical-basic}
If the fibres of \(h\) are connected, then
\[
  d^{\rm v}Q\equiv0
  \quad\Longleftrightarrow\quad
  Q\ \text{is basic}.
\]
\end{proposition}

\begin{proof}
The restriction of \(Q\) to a connected fibre is constant if and only if its derivative
vanishes on every vertical tangent direction.
\end{proof}

For Field,
\begin{equation}
  \left.\partial_xQ\right|_\Psi
  =
  \frac{B+xB'}2I_2
  -
  \frac{A}{2\alpha x^2}\Psi\Psi^\top.
  \label{eq:field-vertical-Q}
\end{equation}
Thus the obstruction to projectability can itself carry latent information. Quadratic
variation is known to reveal hidden states in singular or signal-dependent observation
problems
\citep{JoannidesLeGland1997PerfectObservation,CrisanKouritzinXiong2009SignalDependentObservation,ClarkCrisan2026Identification}.
Filtering with degenerate observation noise has also been treated by
\citet{QianZhangYin2022DegenerateFiltering}. The contribution below is the explicit
recovery/cancellation analysis for the texture--coherence class. For Field the spectral gap already lifts the non-zero observation
fibre; \cref{sec:fixed-shape} identifies the class for which this mechanism is protected.

\section{Fixed-shape spectral recovery}
\label{sec:fixed-shape}

The previous section showed that the instantaneous radar observation leaves a fibre of
latent states, whereas the local quadratic covariance can vary along that fibre.  We now ask
when this variation has a structure strong enough to recover the slow coordinate.

Although the geometry in \cref{sec:model-class} was introduced in a Lamperti coordinate,
it is useful here to allow an arbitrary scalar slow coordinate \(\xi\):
\begin{equation}
  \dd \xi_t=b(\xi_t)\,\dd t+\sigma(\xi_t)\,\dd W_t^0,
  \qquad
  \sigma(\xi)>0.
  \label{eq:slow-general-coordinate}
\end{equation}
The fast covariance is \(D(\xi)\), the coherent observation is
\(Y=a(\xi)Z\), and there is no slow--fast cross quadratic variation.  On a fibre
\(Y=y\), It\^o's formula gives
\begin{equation}
  Q(\xi,y)
  =
  a(\xi)^2D(\xi)
  +
  \sigma(\xi)^2
  \left(\frac{a'(\xi)}{a(\xi)}\right)^2
  yy^\top .
  \label{eq:Q-general-coordinate}
\end{equation}
Accordingly, define
\begin{equation}
  S(\xi)=a(\xi)^2D(\xi),
  \qquad
  q(\xi)
  =
  \sigma(\xi)^2
  \left(\frac{a'(\xi)}{a(\xi)}\right)^2 .
  \label{eq:s-q-general}
\end{equation}
The scalar \(q\) is unchanged by a smooth reparametrization of the slow coordinate.
In Lamperti coordinate, \(\sigma\equiv1\), and \eqref{eq:s-q-general} reduces to the
definition used previously.

\subsection{Fixed-shape class and generalized spectrum}

Assume that the fast covariance has a fixed positive-definite shape:
\begin{equation}
  D(\xi)=d(\xi)D_0,
  \qquad
  d(\xi)>0,
  \qquad
  D_0\in\SPD(m)
  \ \text{constant}.
  \label{eq:fixed-shape-assumption}
\end{equation}
Set
\begin{equation}
  s(\xi)=a(\xi)^2d(\xi).
  \label{eq:s-fixed-shape}
\end{equation}
Then
\begin{equation}
  Q(\xi,y)
  =
  s(\xi)D_0+q(\xi)yy^\top.
  \label{eq:Q-fixed-shape}
\end{equation}

The natural spectrum in this anisotropic setting is the generalized spectrum of the pair
\((Q,D_0)\).  Equivalently, one may whiten once and for all:
\begin{equation}
  \widetilde Q
  =
  D_0^{-1/2}QD_0^{-1/2}
  =
  s(\xi)I_m
  +
  q(\xi)\widetilde y\,\widetilde y^\top,
  \qquad
  \widetilde y=D_0^{-1/2}y .
  \label{eq:Q-whitened}
\end{equation}
Define
\begin{equation}
  \rho^2
  =
  \|\widetilde y\|^2
  =
  y^\top D_0^{-1}y.
  \label{eq:rho-fixed}
\end{equation}

\begin{theorem}[Fixed-shape spectral recovery]
\label{thm:fixed-shape}
Let \(m\ge2\), let \(I\) be an interval of slow states, and assume
\eqref{eq:slow-general-coordinate}--\eqref{eq:fixed-shape-assumption}.  Fix an
observation \(y\neq0\).  Then:

\begin{enumerate}
\item The generalized eigenvalue transverse to \(D_0^{-1}y\) is
\[
  \lambda_\perp(\xi)=s(\xi)
\]
with multiplicity \(m-1\), while the longitudinal generalized eigenvalue is
\[
  \lambda_\parallel(\xi)
  =
  s(\xi)+q(\xi)\rho^2.
\]

\item The generalized spectral gap is
\begin{equation}
  \Delta(\xi;y)
  :=
  \lambda_\parallel-\lambda_\perp
  =
  q(\xi)\rho^2.
  \label{eq:fixed-gap}
\end{equation}

\item Along the observation fibre \(F_y\),
\begin{equation}
  Q(\xi,y)=Q(\xi',y)
  \quad\Longleftrightarrow\quad
  \bigl(s(\xi),q(\xi)\bigr)
  =
  \bigl(s(\xi'),q(\xi')\bigr).
  \label{eq:fixed-injectivity-pair}
\end{equation}
Consequently,
\[
  Q|_{F_y}\ \text{is injective on }I
  \quad\Longleftrightarrow\quad
  \xi\longmapsto\bigl(s(\xi),q(\xi)\bigr)
  \ \text{is injective on }I.
\]

\item If \(q\) itself is injective on \(I\), the slow coordinate is recovered directly from
the gap:
\begin{equation}
  \xi
  =
  q^{-1}\!\left(
  \frac{\Delta(\xi;y)}{\rho^2}
  \right).
  \label{eq:fixed-explicit-inversion}
\end{equation}

\item If \(q\in C^1(I)\) and
\begin{equation}
  |q'(\xi)|\ge c_I>0
  \qquad\text{for all }\xi\in I,
  \label{eq:q-lower-derivative}
\end{equation}
then the inversion is Lipschitz stable:
\begin{equation}
  |\xi-\xi'|
  \le
  \frac{
  |\Delta(\xi;y)-\Delta(\xi';y)|
  }{
  c_I\rho^2
  }.
  \label{eq:fixed-lipschitz}
\end{equation}
\end{enumerate}
\end{theorem}

\begin{proof}
Equation \eqref{eq:Q-whitened} is a rank-one perturbation of a scalar matrix.
Every vector orthogonal to \(\widetilde y\) is therefore an eigenvector with eigenvalue
\(s(\xi)\), whereas
\[
  \widetilde Q\,\widetilde y
  =
  \bigl(s(\xi)+q(\xi)\rho^2\bigr)\widetilde y.
\]
This proves the first two statements.

For \(m\ge2\) and \(y\neq0\), the matrices \(I_m\) and
\(\widetilde y\widetilde y^\top\) are linearly independent.  Hence equality of the two
whitened covariance matrices is equivalent to equality of the two scalar coefficients
\(s\) and \(q\), which proves \eqref{eq:fixed-injectivity-pair}.  If \(q\) is injective,
\eqref{eq:fixed-explicit-inversion} follows from \eqref{eq:fixed-gap}.

Finally, \eqref{eq:q-lower-derivative} implies that \(q'\) has constant sign on \(I\).
The mean-value theorem gives
\[
  |q(\xi)-q(\xi')|
  \ge
  c_I|\xi-\xi'|,
\]
and multiplication by the fixed positive factor \(\rho^2\) gives
\eqref{eq:fixed-lipschitz}.
\end{proof}

The theorem makes an important separation explicit.  The transverse eigenvalue
\(s(\xi)\) contains the scale coupling, whereas the gap isolates \(q(\xi)\).  Thus, whenever
\(q\) is injective, arbitrary variation of the fast scale \(s\) cannot obstruct recovery
through the spectral gap.

\subsection{Local visibility and perturbation stability}

The same structure provides a local criterion.  Along a fixed fibre,
\begin{equation}
  \partial_\xi Q
  =
  s'(\xi)D_0+q'(\xi)yy^\top .
  \label{eq:fixed-vertical-derivative}
\end{equation}
Since the two matrices are linearly independent for \(m\ge2\) and \(y\neq0\),
\begin{equation}
  \partial_\xi Q=0
  \quad\Longleftrightarrow\quad
  s'(\xi)=q'(\xi)=0.
  \label{eq:fixed-local-criterion}
\end{equation}
In particular, \(q'(\xi)\neq0\) is sufficient for a non-zero vertical quadratic defect
regardless of the scale derivative \(s'(\xi)\).  This is the local robustness property that
will fail once the fast covariance is allowed to change shape.

There is also a simple perturbation bound.  Suppose that the estimated whitened covariance is
\[
  \widehat{\widetilde Q}
  =
  \widetilde Q+E,
  \qquad
  \|E\|_{\rm op}\le\varepsilon.
\]
By Weyl's inequality,
\[
  |\widehat\lambda_j-\lambda_j|
  \le\varepsilon,
\]
and therefore
\begin{equation}
  |\widehat\Delta-\Delta|
  \le2\varepsilon.
  \label{eq:gap-weyl}
\end{equation}
Under \eqref{eq:q-lower-derivative}, whenever the perturbed gap remains in the inversion
range \(q(I)\rho^2\),
\begin{equation}
  |\widehat\xi-\xi|
  \le
  \frac{2\varepsilon}{c_I\rho^2}.
  \label{eq:fixed-estimation-bound}
\end{equation}
The denominator already reveals the two intrinsic weak points of gap inversion:
a nearly stationary \(q\) and an observation close to \(y=0\).

\subsection{Field as a globally recoverable member of the class}

We now return to the original texture coordinate \(x\).  In Field,
\[
  a(x)=\sqrt{x},
  \qquad
  \sigma_x^2(x)=\frac{2Ax}{\alpha},
  \qquad
  D(x)=\frac{B(x)}{2}I_2.
\]
Hence
\begin{equation}
  s(x)=\frac{xB(x)}{2},
  \qquad
  q(x)
  =
  \frac{A}{2\alpha x}.
  \label{eq:field-s-q}
\end{equation}
The second function is strictly decreasing on \((0,\infty)\), independently of the positive
function \(B(x)\).

\begin{corollary}[Continuous-noiseless recovery in Field]
\label{cor:field-recovery}
Let \(A>0\), \(\alpha>0\), \(B(x)>0\), and let \(\Psi\neq0\).  In the ideal continuous
experiment in which the local quadratic covariance \(Q(x,\Psi)\) is known,
\begin{equation}
  \Lambda_\theta=\frac{xB(x)}{2},
  \qquad
  \Lambda_\rho-\Lambda_\theta
  =
  \frac{A|\Psi|^2}{2\alpha x},
  \label{eq:field-eigs-corollary}
\end{equation}
and therefore
\begin{equation}
  x
  =
  \frac{A|\Psi|^2}
  {2\alpha(\Lambda_\rho-\Lambda_\theta)}.
  \label{eq:field-x-recovery}
\end{equation}
Once \(x\) is recovered,
\begin{equation}
  \gamma=\frac{\Psi}{\sqrt{x}},
  \qquad
  B(x)=\frac{2\Lambda_\theta}{x}.
  \label{eq:field-gamma-B-recovery}
\end{equation}
Thus the latent texture is globally identified on each non-zero observation fibre,
independently of the functional form of the positive scale coupling \(B(x)\).
\end{corollary}

\begin{proof}
Substitution of \eqref{eq:field-s-q} into
\cref{thm:fixed-shape} with \(D_0=I_2\) gives
\eqref{eq:field-eigs-corollary}.  Since \(q(x)=A/(2\alpha x)\) is injective,
\eqref{eq:field-x-recovery} follows from the gap, and the remaining identities are immediate.
\end{proof}

This corollary is the class-level reinterpretation of the exact Field reconstruction already
identified in the underlying report.  Its scope is narrow but useful: it is a structural
statement about the continuous-noiseless experiment, not a claim that \(x_t\) can be
reconstructed accurately pulse by pulse in a finite-bandwidth radar.

On a compact interval \(I=[x_{\min},x_{\max}]\subset(0,\infty)\),
\[
  |q'(x)|
  =
  \frac{A}{2\alpha x^2}
  \ge
  \frac{A}{2\alpha x_{\max}^2}.
\]
Combining this with \eqref{eq:fixed-estimation-bound} yields
\begin{equation}
  |\widehat x-x|
  \le
  \frac{4\alpha x_{\max}^2}{A|\Psi|^2}\,\varepsilon
  \label{eq:field-stability-bound}
\end{equation}
for a whitened operator-norm covariance error bounded by \(\varepsilon\).
The blow-up as \(|\Psi|\to0\) and the deterioration as \(\alpha\to\infty\) are not numerical
artifacts; they are already present in the deterministic conditioning of the inverse map.

\subsection{An alternative texture model}

The Field result relies on the scaling of the CIR symbol, not merely on the observation
\(\Psi=\sqrt{x}\gamma\).  Consider instead a log-OU texture:
\[
  x=e^{R-v/2},
\]
where \(R\) is an OU process whose stationary variance is \(v\).  Its diffusion coefficient
in the \(x\)-coordinate satisfies
\[
  \sigma_x^2(x)=2Av\,x^2.
\]
Under the same coherent modulation \(a(x)=\sqrt{x}\),
\begin{equation}
  q_{\log{\rm OU}}(x)
  =
  \sigma_x^2(x)
  \left(\frac{a'(x)}{a(x)}\right)^2
  =
  \frac{Av}{2},
  \label{eq:logou-q}
\end{equation}
which is constant.

\begin{corollary}[Failure of gap recovery for log-OU texture]
\label{cor:logou-gap-failure}
For the log-OU texture above, the generalized spectral gap is
\[
  \Delta(x;y)
  =
  \frac{Av}{2}\,y^\top D_0^{-1}y,
\]
and is therefore constant along each non-zero observation fibre.  The texture cannot be
recovered from the gap alone.
\end{corollary}

The constancy of the gap does not exclude information in the transverse scale \(s(x)\).  A complete quadratic-variation counterexample is obtained, for example,
by choosing a fixed-shape fast covariance \(D(x)=d_0x^{-1}D_0\).  Then
\[
  s(x)=a(x)^2d_0x^{-1}=d_0
\]
and \(q(x)=Av/2\), so
\[
  Q(x,y)
  =
  d_0D_0+\frac{Av}{2}yy^\top
\]
is constant on the entire fibre.  Thus multiplicative observation by itself does not guarantee
quadratic recovery of the hidden texture.

This example also clarifies the scope of \cref{thm:fixed-shape}. Quadratic variation has
long been known to reveal hidden-state information in singular observation problems
\citep{CrisanKouritzinXiong2009SignalDependentObservation,ClarkCrisan2026Identification}.
The class-specific contribution here is the exact fixed-shape spectral criterion, its
conditioning bound, and the way it interfaces with the texture--coherence structure.

\section{Shape-changing visibility and exact cancellation}
\label{sec:shape-changing}

The fixed-shape theorem works because, after one global whitening, all variation of the fast
covariance is scalar and the slow contribution appears as a rank-one perturbation.  We now
remove that structural protection.

Let
\begin{equation}
  Q(\xi,y)=S(\xi)+q(\xi)yy^\top,
  \qquad
  S(\xi)\in\SPD(m),
  \label{eq:Q-shape-general}
\end{equation}
with \(y\neq0\) fixed along the observation fibre.  The quantities \(S\) and \(q\) are those
defined in \eqref{eq:s-q-general}.  All derivatives in this section are with respect to the
chosen slow coordinate \(\xi\).

\subsection{Relative scale and shape rates}

At a given \(S\), use the affine-relative inner product
\[
  \langle H,K\rangle_S
  =
  \tr(S^{-1}HS^{-1}K),
  \qquad
  \|H\|_S^2=\langle H,H\rangle_S,
\]
and the traceless projection
\[
  \Pi_S^0H
  =
  H-\frac{\tr(S^{-1}H)}{m}S.
\]
Define
\begin{equation}
  \kappa
  =
  \frac1m\tr(S^{-1}S')
  =
  \frac1m\frac{\dd}{\dd\xi}\log\det S,
  \label{eq:kappa-shape-section}
\end{equation}
and
\begin{equation}
  \chi_{\rm sh}
  =
  \frac1{\sqrt2}\|\Pi_S^0S'\|_S .
  \label{eq:chi-sh-shape-section}
\end{equation}
The scalar \(\kappa\) measures relative scale change; \(\chi_{\rm sh}\) measures change of
normalized shape.

For the observed direction, set
\begin{equation}
  \rho^2
  =
  y^\top S^{-1}y
  \label{eq:rho-shape}
\end{equation}
and introduce the \(S\)-traceless rank-one tensor
\begin{equation}
  H_y
  =
  yy^\top-\frac{\rho^2}{m}S.
  \label{eq:Hy-def}
\end{equation}
A direct whitening calculation gives
\begin{equation}
  \|H_y\|_S
  =
  \rho^2\sqrt{\frac{m-1}{m}}.
  \label{eq:Hy-norm}
\end{equation}

\begin{theorem}[Shape-changing decomposition and cancellation]
\label{thm:shape-changing}
Under \eqref{eq:Q-shape-general}, the vertical derivative of the observed covariance along
the fibre \(y=\mathrm{const}\) is
\[
  Q'=S'+q'yy^\top.
\]
It admits the orthogonal scale--shape decomposition
\begin{equation}
  Q'
  =
  \left(
  \kappa+\frac{q'\rho^2}{m}
  \right)S
  +
  \left(
  \Pi_S^0S'+q'H_y
  \right),
  \label{eq:shape-tensor-decomposition}
\end{equation}
and consequently
\begin{equation}
  \|Q'\|_S^2
  =
  m\left(
  \kappa+\frac{q'\rho^2}{m}
  \right)^2
  +
  \left\|
  \Pi_S^0S'+q'H_y
  \right\|_S^2 .
  \label{eq:shape-decomposition}
\end{equation}
Moreover,
\begin{equation}
  \|Q'\|_S
  \ge
  \left[
  |q'|\rho^2\sqrt{\frac{m-1}{m}}
  -
  \sqrt2\,\chi_{\rm sh}
  \right]_+ ,
  \label{eq:shape-lower-bound}
\end{equation}
and, more sharply,
\begin{equation}
  \|Q'\|_S
  \ge
  \left\{
  m\left(
  \kappa+\frac{q'\rho^2}{m}
  \right)^2
  +
  \left[
  |q'|\rho^2\sqrt{\frac{m-1}{m}}
  -
  \sqrt2\,\chi_{\rm sh}
  \right]_+^2
  \right\}^{1/2}.
  \label{eq:shape-lower-bound-sharp}
\end{equation}
Finally, complete first-order cancellation,
\[
  Q'=0,
\]
occurs if and only if
\begin{equation}
  \Pi_S^0S'=-q'H_y,
  \qquad
  \kappa=-\frac{q'\rho^2}{m}.
  \label{eq:exact-cancellation}
\end{equation}
\end{theorem}

\begin{proof}
Since
\[
  S'=\kappa S+\Pi_S^0S'
\]
and
\[
  yy^\top=\frac{\rho^2}{m}S+H_y,
\]
summing the two decompositions gives
\eqref{eq:shape-tensor-decomposition}.  The first term is proportional to \(S\), while the
second is \(S\)-traceless, so they are orthogonal for
\(\langle\cdot,\cdot\rangle_S\).  This proves
\eqref{eq:shape-decomposition}.

The triangle inequality in the traceless subspace gives
\[
  \|\Pi_S^0S'+q'H_y\|_S
  \ge
  |q'|\|H_y\|_S-\|\Pi_S^0S'\|_S.
\]
Using \eqref{eq:Hy-norm} and
\(\|\Pi_S^0S'\|_S=\sqrt2\,\chi_{\rm sh}\) gives
\eqref{eq:shape-lower-bound}; retaining the orthogonal scale term gives
\eqref{eq:shape-lower-bound-sharp}.

Finally, a sum of two orthogonal components is zero if and only if both components vanish.
This is precisely \eqref{eq:exact-cancellation}.
\end{proof}

The theorem gives a precise meaning to ``shape confounding''.  The rank-one variation
\(q'yy^\top\) itself contains both an isotropic component relative to \(S\) and a traceless
component aligned with the observed direction.  A changing fast covariance can cancel either
component, or both.

\subsection{Why fixed shape is protected}

If the normalized shape is fixed, then
\[
  \chi_{\rm sh}=0,
  \qquad
  \Pi_S^0S'=0.
\]
For \(q'\neq0\) and \(y\neq0\), the traceless contribution \(q'H_y\) cannot vanish when
\(m\ge2\).  Hence
\begin{equation}
  \|Q'\|_S
  \ge
  |q'|\rho^2\sqrt{\frac{m-1}{m}}
  >0.
  \label{eq:fixed-shape-protected}
\end{equation}
Thus a scalar scale variation cannot cancel the directional rank-one term, which accounts
for the local robustness in \cref{thm:fixed-shape}.

Once shape is allowed to vary, that protection disappears.  Complete cancellation requires
a shape rate of magnitude
\begin{equation}
  \chi_{\rm sh}
  =
  |q'|\rho^2
  \sqrt{\frac{m-1}{2m}},
  \label{eq:shape-required-cancellation}
\end{equation}
together with the scale condition in \eqref{eq:exact-cancellation}.  Thus cancellation is not
generic, but it is allowed by the geometry and can occur exactly.

\subsection{Local cancellation under unconstrained shape change}

The preceding statement yields the following local cancellation result.

\begin{corollary}[Local cancellation]
\label{cor:shape-cancellation}
Fix a point \((\xi_0,y)\) with \(y\neq0\), a positive-definite matrix
\(S_0=S(\xi_0)\), and any prescribed value \(q'(\xi_0)\), including
\(q'(\xi_0)\neq0\).  There exists a smooth local path
\(S(\xi)\in\SPD(m)\), with \(S(\xi_0)=S_0\), such that
\[
  \left.\partial_\xi Q(\xi,y)\right|_{\xi=\xi_0}=0.
\]
Consequently, no positive lower bound on local quadratic visibility can depend on
\(q'(\xi_0)\) and \(y\) alone when shape variation is unconstrained.
\end{corollary}

\begin{proof}
Choose
\[
  S'(\xi_0)=-q'(\xi_0)yy^\top.
\]
Then \(Q'(\xi_0,y)=0\).  Since \(\SPD(m)\) is open in the vector space of symmetric
matrices, the affine path
\[
  S(\xi)=S_0+(\xi-\xi_0)S'(\xi_0)
\]
remains positive definite on a sufficiently small interval around \(\xi_0\).
\end{proof}

The corollary is intentionally local.  A zero first derivative at one point does not imply
global indistinguishability of two finite latent states, nor does it exclude information in
\(Lh\), higher iterated covariations, or the complete observation history.  It says only that
the first quadratic signature can be erased by admissible shape variation.

\subsection{An intrinsic normalized visibility}

The norm \(\|Q'\|_S\) depends on the speed at which the fibre is parametrized by the chosen
slow coordinate.  We therefore normalize by the latent diffusion metric.

In the general coordinate \(\xi\), the block-diagonal latent metric is
\begin{equation}
  g
  =
  \frac{\dd\xi^2}{\sigma(\xi)^2}
  +
  \dd z^\top D(\xi)^{-1}\dd z.
  \label{eq:latent-metric-general-coordinate}
\end{equation}
Along the fibre \(y=a(\xi)z\), a tangent corresponding to unit change in \(\xi\) is
\[
  V
  =
  \partial_\xi
  -
  \frac{a'}{a}z^i\partial_{z^i}.
\]
Using \(S=a^2D\) and \(\rho^2=y^\top S^{-1}y\),
\begin{equation}
  \|V\|_g^2
  =
  \frac{1}{\sigma^2}
  +
  \left(\frac{a'}{a}\right)^2\rho^2
  =
  \frac{1+q\rho^2}{\sigma^2}.
  \label{eq:vertical-vector-norm}
\end{equation}
Hence the unit vertical vector is
\[
  U
  =
  \frac{\sigma}{\sqrt{1+q\rho^2}}\,V,
\]
and the normalized vertical quadratic visibility is
\begin{equation}
  \mathfrak v_Q
  :=
  \|d^{\rm v}Q(U)\|_S
  =
  \frac{\sigma}{\sqrt{1+q\rho^2}}
  \|Q'\|_S .
  \label{eq:intrinsic-visibility}
\end{equation}
In Lamperti coordinate \(\sigma=1\).  Unlike \(\|Q'\|_S\) alone,
\(\mathfrak v_Q\) is invariant under a smooth reparametrization of the one-dimensional slow
coordinate.

Combining \eqref{eq:intrinsic-visibility} with
\eqref{eq:shape-lower-bound-sharp} gives an intrinsic sufficient lower bound for local
quadratic visibility.  This quantity will also make explicit, in
\cref{sec:resolvability}, the distinction between structural visibility and finite-sample
resolvability.

\subsection{Two-dimensional interference geometry}

For coherent I/Q data, let \(m=2\) and write, in the instantaneous principal-axis frame,
\begin{equation}
  S
  =
  s\,
  R_\theta
  \begin{pmatrix}
  e^\eta&0\\
  0&e^{-\eta}
  \end{pmatrix}
  R_\theta^\top.
  \label{eq:S-2D-shape-section}
\end{equation}
As in \cref{sec:shape-geometry},
\[
  \kappa=\frac{s'}{s},
  \qquad
  \bm c
  =
  \bigl(\eta',\,2\theta'\sinh\eta\bigr),
  \qquad
  \|\bm c\|=\chi_{\rm sh}.
\]
Let
\[
  u=S^{-1/2}y
  =
  \rho(\cos\phi,\sin\phi),
  \qquad
  \delta=\frac{q'\rho^2}{2},
\]
and
\[
  \bm n_\phi=(\cos2\phi,\sin2\phi).
\]
Then the whitened vertical derivative
\[
  W=S^{-1/2}Q'S^{-1/2}
\]
satisfies
\begin{equation}
  \|W\|_F^2
  =
  2(\kappa+\delta)^2
  +
  2\|\bm c+\delta\bm n_\phi\|^2.
  \label{eq:2D-interference}
\end{equation}

Equation \eqref{eq:2D-interference} turns shape confounding into a planar vector-addition
problem.  If \(\bm c\) is aligned with \(+\bm n_\phi\), shape change reinforces the
rank-one signature.  If it is orthogonal, the two contributions add in quadrature.  If
\(\bm c=-\delta\bm n_\phi\), the traceless anisotropic contribution cancels exactly.
Complete cancellation further requires
\begin{equation}
  \kappa=-\delta.
  \label{eq:2D-full-cancellation}
\end{equation}

For the dimensionless local benchmark
\[
  \rho^2=1,\qquad q'=1,\qquad \delta=\frac12,
\]
the five cases
\[
  \text{fixed shape},\quad
  \text{reinforcement},\quad
  \text{orthogonal},\quad
  \text{anisotropic cancellation},\quad
  \text{complete cancellation}
\]
give respectively
\begin{equation}
  \|W\|_F
  =
  1,\quad
  \sqrt{\frac52},\quad
  \sqrt{\frac32},\quad
  \frac1{\sqrt2},\quad
  0.
  \label{eq:2D-benchmark-values}
\end{equation}
These values will be used as exact checks in the numerical section.

\subsection{Interpretation}

For fixed normalized shape, one global whitening transformation separates scale from the
rank-one directional term, and the spectral gap cannot be cancelled by scalar scale variation.
When the normalized shape varies, its traceless derivative can interfere with that term and,
under the conditions of \eqref{eq:exact-cancellation}, cancel it exactly.  These statements
concern the observed quadratic tensor only; information may still be present in the drift,
higher stochastic characteristics or the full observation history.  The next section turns to
the separate issue of whether a structurally present quadratic signature can be resolved from
finite-rate noisy data.

\section{Radar resolvability under finite sampling}
\label{sec:resolvability}

The preceding theorems concern an ideal continuous experiment in which the quadratic
variation of the coherent return is available without receiver corruption. A radar instead
provides finite-rate samples through a noisy, bandwidth-limited receiver. Structural
observability must therefore be separated from the temporal and noise resolution of a local
covariance estimate.

\subsection{Local covariance and finite sampling}

Let
\begin{equation}
  \dd Y_t=\beta_t\dd t+G_t\dd W_t,
  \qquad Q_t=G_tG_t^\top.
  \label{eq:observed-local-SDE}
\end{equation}
At sampling interval \(\Delta\), with \(K\) increments in a window \(h=K\Delta\), consider
\begin{equation}
  \widehat Q_{t,h}
  =
  \frac1{K\Delta}
  \sum_{k=0}^{K-1}\Delta_kY\,\Delta_kY^\top,
  \label{eq:local-realized-Q}
\end{equation}
whose dynamic target is
\begin{equation}
  \overline Q_{t,h}
  =
  \frac1h\int_t^{t+h}Q_s\,\dd s.
  \label{eq:window-average-Q}
\end{equation}
A useful local regime is
\begin{equation}
  \Delta\ll h\ll\tau_Q,
  \label{eq:time-scale-separation}
\end{equation}
where \(\tau_Q\) is the shortest relevant time scale of \(Q_t\). For Field, a conservative
choice is \(h\ll\min(A^{-1},2/B)\). If \(\widehat Q=Q+E\), Weyl's inequality gives
\begin{equation}
  |\widehat\delta_Q-\delta_Q|
  \le2\|E\|_{\rm op}.
  \label{eq:weyl-gap-resolution}
\end{equation}

\subsection{Independent sample noise}

For
\begin{equation}
  Y_k^{\rm obs}=Y_{t_k}+\varepsilon_k,
  \qquad
  \varepsilon_k\sim\mathcal N(0,R_n)\quad\text{iid},
  \label{eq:A0-observation}
\end{equation}
the differenced noise is MA(1), and the naive covariance contains
\begin{equation}
  Q_{t_k}+\frac{2R_n}{\Delta}+O(\Delta).
  \label{eq:A0-divergence}
\end{equation}
The \(1/\Delta\) divergence is specific to iid per-sample noise.

For isotropic \(R_n=nI_2\), the leading bias is isotropic but the gap variance increases.

\begin{proposition}[Frozen Gaussian variance of the raw gap estimator]
\label{prop:A0-gap-variance}
Assume that \(Q\) is constant over the window and diagonal in a fixed
radial--tangential frame, with eigenvalues \(\lambda_\rho,\lambda_\theta\). For
\begin{equation}
  \widehat\delta
  =
  \frac1{K\Delta}
  \sum_{k=0}^{K-1}
  \left[
  (\Delta Y^{\rm obs}_{\rho,k})^2
  -
  (\Delta Y^{\rm obs}_{\theta,k})^2
  \right],
  \label{eq:raw-gap-estimator}
\end{equation}
let \(c_j=\lambda_j\Delta+2n\). Then
\begin{equation}
  \Var(\widehat\delta)
  =
  \frac{2}{K^2\Delta^2}
  \left[
  K(c_\rho^2+c_\theta^2)+4(K-1)n^2
  \right].
  \label{eq:A0-gap-var-exact}
\end{equation}
For \(K=h/\Delta\gg1\),
\begin{equation}
  \Var(\widehat\delta)
  \simeq
  \frac2h
  \left[
  \frac{12n^2}{\Delta}
  +4n(\lambda_\rho+\lambda_\theta)
  +(\lambda_\rho^2+\lambda_\theta^2)\Delta
  \right],
  \label{eq:A0-gap-var-asymptotic}
\end{equation}
whose variance-optimal interval at fixed \(h\) is
\begin{equation}
  \Delta_\star
  =
  \frac{2\sqrt3\,n}
  {\sqrt{\lambda_\rho^2+\lambda_\theta^2}}.
  \label{eq:A0-delta-star}
\end{equation}
\end{proposition}

The derivation is given in \cref{app:numerical-protocol}. This optimum is specific to the
frozen local model. For constant \(B\), an OU calculation gives the complementary
localization bias
\begin{equation}
  \operatorname{Bias}_h
  =
  \frac{A}{2\alpha}
  (1-|\gamma_0|^2)
  \left[
  1-\frac{1-e^{-Bh}}{Bh}
  \right]
  =
  \frac{ABh}{4\alpha}(1-|\gamma_0|^2)+O(h^2).
  \label{eq:field-window-bias}
\end{equation}

\subsection{Correlated and anisotropic receiver noise}
\label{subsec:correlated-receiver}

Let
\begin{equation}
  C_\ell=\E[\varepsilon_{k+\ell}\varepsilon_k^\top].
  \label{eq:receiver-C-lag}
\end{equation}
Then
\begin{equation}
  N_\ell
  :=
  \E[\eta_{k+\ell}\eta_k^\top]
  =
  2C_\ell-C_{\ell+1}-C_{\ell-1}.
  \label{eq:receiver-increment-cov}
\end{equation}
For a vector AR(1) receiver with parameter \(\rho\) and stationary covariance \(R_0\),
\begin{align}
  N_0&=2(1-\rho)R_0,
  \label{eq:AR1-N0}\\
  N_\ell&=-(1-\rho)^2\rho^{\ell-1}R_0,\qquad \ell\ge1.
  \label{eq:AR1-Nlag}
\end{align}

Write the physical and receiver anisotropies as
\[
  Q_{\rm phys}
  =
  \bar\lambda I_2+\frac\delta2A(\phi),
  \qquad
  \frac{N_0}{\Delta}
  =
  \nu I_2+\frac\beta2A(\phi_n).
\]
Then
\begin{equation}
  \delta_{\rm meas}^2
  =
  \delta^2+\beta^2
  +2\delta\beta\cos2(\phi_n-\phi).
  \label{eq:instrument-gap-vector-law}
\end{equation}
Receiver anisotropy can therefore reinforce, rotate or cancel the physical gap, making I/Q
calibration structurally relevant
\citep{ChengEtAl2014IQImbalance,Cardillo2026IQCalibration}.

If the continuous-time receiver variogram satisfies
\(V_\varepsilon(\Delta)\sim K\Delta^\beta\), then
\(V_\varepsilon(\Delta)/\Delta\) diverges for \(\beta<1\), tends to a finite instrumental
quadratic variation for \(\beta=1\), and vanishes for \(\beta>1\). In the critical case,
quadratic variation identifies only \(Q_{\rm sea}+Q_{\rm rec}\) without calibration.

Pre-averaging and realized-kernel methods address such noise contamination
\citep{JacodEtAl2009Preaveraging,BarndorffNielsenEtAl2008RealizedKernels,BarndorffNielsenEtAl2011MultivariateKernels,Bibinger2011Covariance,ParkHongLinton2016FourierKernel}.
Here they are used only as benchmarks because the local target itself changes over the
smoothing window. Consequently, geometric observability does not by itself imply either
finite-sample resolvability or stable latent reconstruction.

\section{Controlled numerical benchmarks}
\label{sec:numerics}

The computations test the analytic identities and their finite-sample conditioning; they do
not constitute validation against historical radar data, which are unavailable here. The
canonical Field simulations use exact CIR and OU transitions.

\subsection{Geometric identities and an alternative texture model}

For a non-diagonal fixed \(D_0\), the generalized spectral gap agrees with
\[
  \Delta(\xi;y)=q(\xi)y^\top D_0^{-1}y
\]
to machine precision. Inverting \(q\) gives
\[
  \max|\widehat\xi-\xi|=7.6\times10^{-15},
\]
while a perturbed-covariance test produces a latent error \(2.36\times10^{-3}\), below the
deterministic bound \(3.98\times10^{-3}\).

The two-dimensional shape-changing benchmark reproduces the exact values in
\cref{tab:shape-benchmark}.

\begin{table}[t]
\centering
\caption{Exact two-dimensional shape-interference benchmark.}
\label{tab:shape-benchmark}
\begin{tabular}{lc}
\toprule
Configuration & \(\|W\|_F\)\\
\midrule
Fixed shape & \(1.000000\)\\
Reinforcement & \(1.581139=\sqrt{5/2}\)\\
Orthogonal shape variation & \(1.224745=\sqrt{3/2}\)\\
Anisotropic cancellation & \(0.707107=1/\sqrt2\)\\
Complete cancellation & \(0\)\\
\bottomrule
\end{tabular}
\end{table}

\begin{figure}[t]
  \centering
  \includegraphics[width=0.68\linewidth]{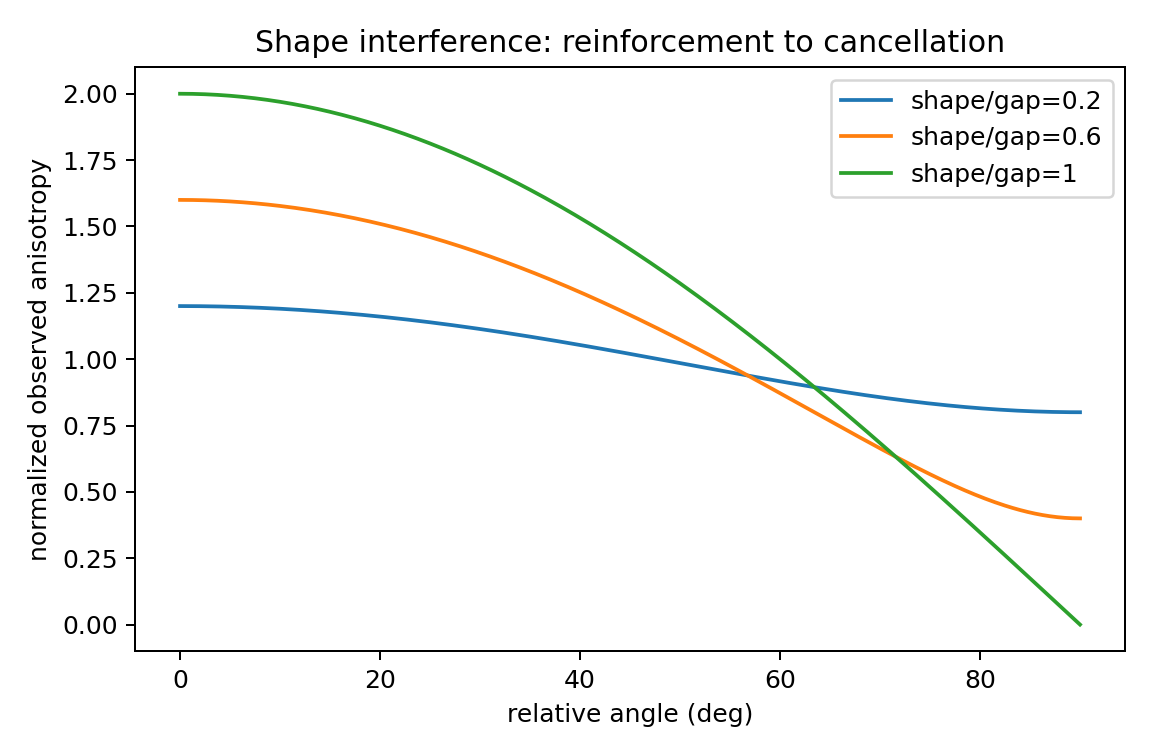}
  \caption{Two-dimensional shape interference. Relative orientation can reinforce or suppress
  the observed anisotropy, with exact cancellation in the anti-aligned case.}
  \label{fig:shape-interference}
\end{figure}

For the same observation \(a(x)=\sqrt x\), CIR texture gives
\(q_{\rm CIR}=A/(2\alpha x)\), whereas the log-OU model gives
\(q_{\log{\rm OU}}=Av/2\). The implementation reproduces both expressions, confirming that
gap recovery relies on the slow diffusion symbol rather than on multiplicative observation
alone.

\subsection{Finite PRF and dynamic localization}

The frozen A0 benchmark uses
\[
  \lambda_\theta=1,\qquad
  \lambda_\rho=1.5,\qquad
  n=10^{-3},\qquad h=0.2,
\]
with \(30\,000\) Monte Carlo realizations. Formula \eqref{eq:A0-delta-star} predicts
\(\Delta_\star=1.9215\times10^{-3}\), and theoretical and Monte Carlo standard deviations
agree to about one percent over the tested grid.

\begin{figure}[t]
  \centering
  \includegraphics[width=0.70\linewidth]{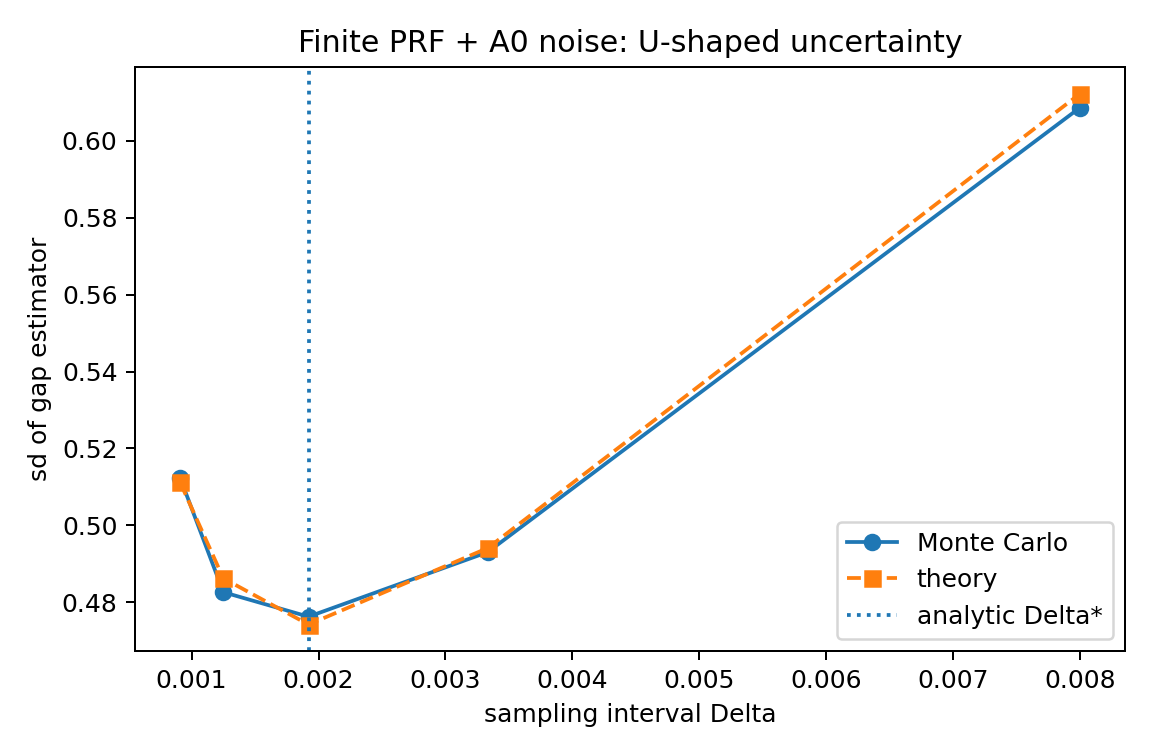}
  \caption{Frozen A0 benchmark. At fixed window \(h\), the uncertainty is U-shaped in
  \(\Delta\), with the analytic optimum marked.}
  \label{fig:A0-prf}
\end{figure}

The full dynamic Field test uses
\[
  A=1,\quad B=4,\quad \alpha=4,\quad
  x_0=1,\quad \gamma_0=(1,0.25),\quad
  \Delta=10^{-3},
\]
and \(12\,000\) exact CIR--OU trajectories.

\begin{table}[t]
\centering
\caption{Dynamic Field benchmark.}
\label{tab:dynamic-field}
\begin{tabular}{rrrrr}
\toprule
\(h\) & \(K\) & mean \(\widehat\delta\) &
mean exact gap & sd\((\widehat\delta)\)\\
\midrule
0.02 & 20  & 0.1427 & 0.1323 & 0.9281\\
0.05 & 50  & 0.1413 & 0.1321 & 0.5932\\
0.10 & 100 & 0.1370 & 0.1317 & 0.4196\\
0.20 & 200 & 0.1342 & 0.1313 & 0.3025\\
\bottomrule
\end{tabular}
\end{table}

The trajectory-level correlations are \(0.008\), \(0.046\), \(0.117\) and \(0.197\).
Thus the ideal quadratic tensor is informative while the raw local estimator remains weak on
individual trajectories.

\subsection{Receiver correlation and robustification crossover}

For the AR(1) receiver with \(\rho=0.6\), empirical increment-noise covariances agree with
\eqref{eq:AR1-N0}--\eqref{eq:AR1-Nlag} at relative Frobenius errors
\(0.19\%\), \(0.50\%\), \(2.01\%\) and \(2.71\%\) for lags \(0,\ldots,3\).

We compare the raw estimator with overlapping pre-averaging under iid noise
\citep{JacodEtAl2009Preaveraging} and a local Parzen realized kernel under correlated noise
\citep{BarndorffNielsenEtAl2008RealizedKernels,BarndorffNielsenEtAl2011MultivariateKernels}.
The tuning values in \cref{tab:noise-crossover} are selected \emph{ex post} from finite grids.

\begin{table}[t]
\centering
\caption{Finite-sample noise crossover.}
\label{tab:noise-crossover}
\begin{tabular}{lrrrr}
\toprule
Model & Noise scale & Raw RMSE & Best robust RMSE & Tuning\\
\midrule
iid & \(2.5\times10^{-4}\) & 0.381 & 0.464 & \(k=5\)\\
iid & \(1.0\times10^{-3}\) & 0.615 & 0.514 & \(k=5\)\\
iid & \(4.0\times10^{-3}\) & 1.705 & 0.733 & \(k=7\)\\
iid & \(1.0\times10^{-2}\) & 3.675 & 0.921 & \(k=10\)\\
\midrule
AR(1) & \(0.25\,R_0\) & 0.307 & 0.326 & \(H=1\)\\
AR(1) & \(1\,R_0\)    & 0.358 & 0.369 & \(H=1\)\\
AR(1) & \(4\,R_0\)    & 0.776 & 0.674 & \(H=3\)\\
AR(1) & \(10\,R_0\)   & 1.766 & 0.923 & \(H=10\)\\
AR(1) & \(25\,R_0\)   & 4.228 & 1.200 & \(H=20\)\\
\bottomrule
\end{tabular}
\end{table}

\begin{figure}[t]
  \centering
  \includegraphics[width=0.70\linewidth]{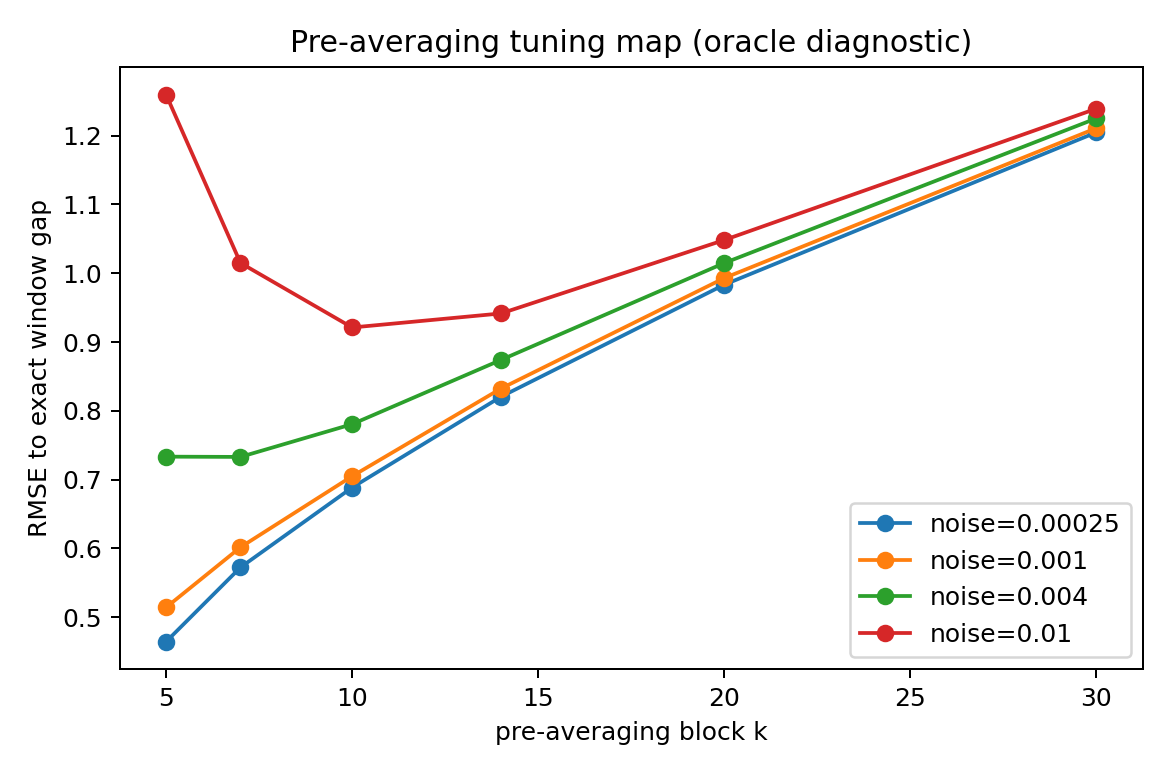}
  \caption{Pre-averaging crossover in the dynamic Field benchmark. Smoothing is detrimental
  at the lowest noise level but becomes advantageous as receiver noise increases.}
  \label{fig:preaverage-crossover}
\end{figure}

There is no uniform finite-sample dominance of the noise-robust estimators in this local
dynamic regime. The outcome depends jointly on noise level, localization window and
smoothing scale.

\section{Conclusion}
\label{sec:conclusion}

This paper started from a simple mismatch between the latent model and the radar
measurement.  In texture--coherence diffusions, the physically useful slow and fast
variables are not observed separately: a coherent radar sees only their multiplicative
combination \(Y=a(r)Z\).  The instantaneous observation therefore defines a fibre of latent
states rather than a unique state.  The central question was whether the stochastic
evolution of the observed field can recover information that is absent from a single
instantaneous sample.

The answer is conditional.  In the fixed-shape class
\(D(r)=d(r)D_0\), one global whitening separates the fast covariance scale from the
rank-one contribution generated by the slow variable.  The resulting generalized spectral
gap,
\[
  \Delta(r;y)=q(r)\,y^\top D_0^{-1}y,
\]
provides an exact recovery mechanism whenever the pair \((s,q)\) is injective along the
observation fibre, and \(q\) alone is sufficient whenever it is injective.  Field's CIR--OU
model belongs to this favourable class: away from zero observation, its texture is recovered
from the radial--transverse gap independently of the positive fast scale \(B(x)\).
The log-OU example shows that this is not a generic consequence of multiplicative
observation; it depends on the scaling of the slow diffusion symbol.

Allowing the fast covariance to change shape modifies the picture qualitatively.  The
vertical variation of the observed covariance splits orthogonally into scale and traceless
shape components.  This yields both a lower bound on local quadratic visibility and exact
conditions under which the shape variation cancels the rank-one slow signature.  Hence
fixed shape is a protected class, whereas unconstrained shape change admits local
first-order cancellation.  These statements concern the quadratic tensor only: they do not
exclude information carried by the drift, higher stochastic characteristics or the complete
observation history.

The acquisition analysis adds a second qualification. Structural observability in the
continuous noiseless experiment does not imply finite-sample resolvability. A local radar
estimate must balance sampling variance against the evolution of the physical covariance
within the localization window, and receiver anisotropy can reinforce, rotate or cancel the
same spectral gap used for latent recovery. The controlled numerical benchmarks confirm
the exact geometric identities and the finite-PRF variance formula, but also show that
trajectory-level reconstruction can remain poorly conditioned and that noise-robust
smoothing does not dominate uniformly at finite sample size. In practice, identification of
a latent coordinate, resolution of its quadratic signature and stable inversion are separate
questions. The distinction remains relevant when the Field benchmark is replaced by more
complete sea-clutter or receiver models.

\appendix
\section{Tensorial details and supplementary proofs}
\label{app:tensorial-proofs}

\subsection{Intrinsic vertical derivative}

Let \(h:M\to N\) be a submersion and \(Q\) a section of
\(h^\ast\Sym^2(TN)\).  For \(v\in\ker Dh_x\), choose a curve
\(c(s)\subset h^{-1}(h(x))\) with \(c(0)=x\) and \(\dot c(0)=v\).
Since \(h(c(s))\) is constant, all \(Q_{c(s)}\) lie in the same vector
space \(\Sym^2(T_{h(x)}N)\).  Hence
\[
  d^{\rm v}Q_x(v)
  =
  \left.\frac{\dd}{\dd s}Q_{c(s)}\right|_{s=0}
\]
is defined without choosing a connection on \(N\).

If the fibres are connected, \(d^{\rm v}Q\equiv0\) if and only if \(Q\)
is constant on each fibre and therefore basic.  This proves
\cref{prop:vertical-basic}.

\subsection{Coordinate invariance of the scalar coefficient \texorpdfstring{\(q\)}{q}}

Suppose \(\eta=f(\xi)\) is a smooth monotone reparametrization of the slow coordinate.
The transformed diffusion coefficient is
\[
  \sigma_\eta(\eta)
  =
  f'(\xi)\sigma_\xi(\xi),
\]
whereas
\[
  \frac{\partial_\eta a}{a}
  =
  \frac{1}{f'(\xi)}
  \frac{\partial_\xi a}{a}.
\]
Therefore
\[
  \sigma_\eta^2
  \left(\frac{\partial_\eta a}{a}\right)^2
  =
  \sigma_\xi^2
  \left(\frac{\partial_\xi a}{a}\right)^2,
\]
so the coefficient \(q\) in \eqref{eq:s-q-general} is a scalar function of the latent
state rather than an artifact of the chosen one-dimensional coordinate.

\subsection{Affine-relative scale--shape decomposition}

For \(S\in\SPD(m)\), define
\[
  \langle H,K\rangle_S
  =
  \tr(S^{-1}HS^{-1}K).
\]
Under a nonsingular linear change of observed coordinates,
\[
  S\mapsto ASA^\top,\qquad
  H\mapsto AHA^\top,
\]
one has
\[
  \tr\!\left[
  (ASA^\top)^{-1}(AHA^\top)
  (ASA^\top)^{-1}(AKA^\top)
  \right]
  =
  \tr(S^{-1}HS^{-1}K),
\]
so this inner product is congruence invariant.

The decomposition
\[
  H
  =
  \frac{\tr(S^{-1}H)}{m}S
  +
  \Pi_S^0H
\]
is orthogonal because
\[
  \left\langle
  S,\Pi_S^0H
  \right\rangle_S
  =
  \tr(S^{-1}\Pi_S^0H)
  =
  0.
\]
Applying the decomposition to \(S'\) and \(yy^\top\) gives
\cref{eq:shape-tensor-decomposition}.

\subsection{Norm of the directional traceless tensor}

Let
\[
  u=S^{-1/2}y,
  \qquad
  \rho^2=u^\top u.
\]
Then
\[
  S^{-1/2}H_yS^{-1/2}
  =
  uu^\top-\frac{\rho^2}{m}I.
\]
Using
\[
  (uu^\top)^2=\rho^2uu^\top
\]
gives
\begin{align*}
  \|H_y\|_S^2
  &=
  \tr\left[
  \left(
  uu^\top-\frac{\rho^2}{m}I
  \right)^2
  \right]\\
  &=
  \rho^4
  -\frac{2\rho^4}{m}
  +\frac{\rho^4}{m}\\
  &=
  \rho^4\frac{m-1}{m},
\end{align*}
which proves \eqref{eq:Hy-norm}.

\subsection{Unit vertical direction}

In the general slow coordinate \(\xi\), the block-diagonal diffusion metric is
\[
  g
  =
  \frac{\dd\xi^2}{\sigma(\xi)^2}
  +
  \dd z^\top D(\xi)^{-1}\dd z.
\]
Along a fibre \(y=a(\xi)z\),
\[
  V
  =
  \partial_\xi-\frac{a'}{a}z^i\partial_{z^i}.
\]
Since
\[
  z^\top D^{-1}z
  =
  y^\top S^{-1}y
  =
  \rho^2,
\]
one obtains
\[
  \|V\|_g^2
  =
  \frac{1}{\sigma^2}
  +
  \left(\frac{a'}{a}\right)^2\rho^2
  =
  \frac{1+q\rho^2}{\sigma^2}.
\]
Thus
\[
  U
  =
  \frac{\sigma}{\sqrt{1+q\rho^2}}V
\]
is the unit vertical direction and
\[
  \|d^{\rm v}Q(U)\|_S
  =
  \frac{\sigma}{\sqrt{1+q\rho^2}}\|Q'\|_S,
\]
as used in \eqref{eq:intrinsic-visibility}.

\subsection{Two-dimensional form}

For \(m=2\), write
\[
  S
  =
  sR_\theta
  \diag(e^\eta,e^{-\eta})
  R_\theta^\top.
\]
In the instantaneous eigenbasis,
\[
  S^{-1/2}S'S^{-1/2}
  =
  \frac{s'}{s}I
  +
  \begin{pmatrix}
  \eta'&2\theta'\sinh\eta\\
  2\theta'\sinh\eta&-\eta'
  \end{pmatrix}.
\]
If
\[
  u=S^{-1/2}y
  =
  \rho(\cos\phi,\sin\phi),
\]
then
\[
  uu^\top-\frac{\rho^2}{2}I
  =
  \frac{\rho^2}{2}
  \begin{pmatrix}
  \cos2\phi&\sin2\phi\\
  \sin2\phi&-\cos2\phi
  \end{pmatrix}.
\]
Substitution into the whitened derivative
\[
  W
  =
  S^{-1/2}Q'S^{-1/2}
\]
gives \eqref{eq:2D-interference}.

\section{Acquisition formulas and numerical protocol}
\label{app:numerical-protocol}

\subsection{Exact transitions used in the Field benchmark}

The dynamic Field simulations use exact one-step transitions.  For the CIR texture
\[
  \dd x_t=A(1-x_t)\dd t+
  \sqrt{\frac{2A}{\alpha}x_t}\,\dd W_t,
\]
the transition over \(\Delta\) is sampled through the standard non-central chi-square law.
Each OU quadrature is propagated exactly as
\[
  \gamma_{k+1}
  =
  e^{-B\Delta/2}\gamma_k
  +
  \sqrt{\frac{1-e^{-B\Delta}}{2}}\,\xi_k,
  \qquad
  \xi_k\sim\mathcal N(0,1).
\]
Euler discretization is therefore not used for the canonical numerical benchmark.

\subsection{Derivation of the frozen A0 gap variance}

Under the assumptions of \cref{prop:A0-gap-variance}, let
\[
  X_{\rho,k}
  =
  \Delta Y_{\rho,k}^{\rm obs},
  \qquad
  X_{\theta,k}
  =
  \Delta Y_{\theta,k}^{\rm obs}.
\]
For either component \(j\in\{\rho,\theta\}\),
\[
  \Var(X_{j,k})=c_j=\lambda_j\Delta+2n.
\]
The physical Brownian increments are independent across \(k\), whereas the differenced
sample noise has
\[
  \Cov(X_{j,k+1},X_{j,k})=-n
\]
and zero covariance beyond lag one.  For centered jointly Gaussian variables,
\[
  \Cov(X^2,Y^2)=2\Cov(X,Y)^2.
\]
Hence
\begin{align}
  \Var\left(\sum_{k=0}^{K-1}X_{j,k}^2\right)
  &=
  2Kc_j^2+4(K-1)n^2.
  \label{eq:app-single-component-square-var}
\end{align}
The radial and transverse components are independent under isotropic receiver noise and the
frozen diagonal physical covariance, so
\[
  \Var(\widehat\delta)
  =
  \frac{1}{K^2\Delta^2}
  \left[
  2K(c_\rho^2+c_\theta^2)
  +
  8(K-1)n^2
  \right],
\]
which is \eqref{eq:A0-gap-var-exact}.

With \(K=h/\Delta\gg1\),
\begin{align*}
  c_\rho^2+c_\theta^2
  &=
  (\lambda_\rho^2+\lambda_\theta^2)\Delta^2
  +
  4n(\lambda_\rho+\lambda_\theta)\Delta
  +
  8n^2.
\end{align*}
After including the lag-one term, the coefficient of \(n^2\) becomes \(12\), yielding
\eqref{eq:A0-gap-var-asymptotic}.  Differentiation with respect to \(\Delta\) at fixed \(h\)
gives
\[
  \Delta_\star
  =
  \frac{2\sqrt3\,n}
  {\sqrt{\lambda_\rho^2+\lambda_\theta^2}}.
\]

\subsection{Correlated receiver increments}

Let
\[
  C_\ell
  =
  \E[\varepsilon_{k+\ell}\varepsilon_k^\top]
\]
and
\[
  \eta_k=\varepsilon_{k+1}-\varepsilon_k.
\]
Then
\begin{align*}
  \E[\eta_{k+\ell}\eta_k^\top]
  &=
  \E[
  (\varepsilon_{k+\ell+1}-\varepsilon_{k+\ell})
  (\varepsilon_{k+1}-\varepsilon_k)^\top]\\
  &=
  C_\ell-C_{\ell+1}-C_{\ell-1}+C_\ell,
\end{align*}
which proves \eqref{eq:receiver-increment-cov}.  For
\(C_\ell=\rho^{|\ell|}R_0\), one obtains
\eqref{eq:AR1-N0}--\eqref{eq:AR1-Nlag}.

\subsection{Reproducibility record}

For each numerical statement reported in the main text, the archived scripts record:
\begin{itemize}
\item dimensionless parameters \(A\Delta\), \(B\Delta/2\), \(\alpha\), local window \(h\),
      and receiver-noise scale;
\item random seed and Monte Carlo count;
\item estimator definition and tuning rule;
\item whether tuning is operational or oracle/ex-post;
\item target quantity, distinguishing instantaneous \(Q_t\) from the exact window average;
\item Monte Carlo uncertainty or a directly reproducible summary statistic.
\end{itemize}

The numerical programme uses controlled alternative models because the historical radar data
are unavailable. These include non-CIR texture, same-symbol drift misspecification,
temporally correlated receiver noise and anisotropic I/Q contamination. The calculations test
the structural results and their range of validity; they are not presented as experimental
validation.


\section*{Acknowledgements}
The authors sincerely thank Ismaël Bailleul for kindly providing the arXiv endorsement that made the dissemination of this work possible.

\bibliographystyle{plainnat}
\bibliography{references}

\end{document}